\documentclass[11pt]{article}
\usepackage{amsmath,amssymb,amsthm,amscd,dsfont}
\usepackage{amsfonts,latexsym,rawfonts,amsmath,amssymb,amsthm}
\usepackage{lscape}
\usepackage{amscd, float,times,rotating}
\usepackage{pb-diagram}
\usepackage{hyperref}
\numberwithin{equation}{section}

\newtheorem{prop}{Proposition}[section]
\newtheorem{theorem}[prop]{Theorem}
\newtheorem{lemma}[prop]{Lemma}
\newtheorem{corollary}[prop]{Corollary}
\newtheorem{remark}[prop]{Remark}
\newtheorem{example}[prop]{Example}
\newtheorem{definition}[prop]{Definition}

\begin{document}
\title {Willmore-type inequality for closed hypersurfaces in complete manifolds with Ricci curvature bounded  below}
\author{Xiaoshang Jin   \thanks{X. J. is supported by the NSFC (Grant No. 12201225).}
\ \ \ \ \ \ Jiabin Yin	\thanks{J. Y. is supported by the NSFC (Grant No. 12201138) and Mathematics Tianyuan fund project (Grant No. 12226350).}}

\date{}
\maketitle
\begin{abstract}
In this paper, we establish a Willmore-type inequality for closed hypersurfaces in a complete Riemannian manifold of dimension $n+1$ with ${\rm Ric}\geq-ng$. It extends the classic result of Argostianiani, Fogagnolo and Mazzieri in \cite{agostiniani2020sharp} to the Riemannian manifold of negative curvature. As an application, we construct a Willmore-type inequality for closed hypersurfaces in hyperbolic space and obtain the characterization of geodesic sphere.
\end{abstract}

\section{Introduction}  \label{sect1}
The classical Willmore inequality \cite{willmore1968mean} for a bounded domain $\Omega$ of $\mathbb R^3$ with smooth boundary
says that
\begin{equation*}
\int_{\partial\Omega}H^2\geq16\pi,
\end{equation*}
where $H$ is mean curvature of $\partial\Omega$.  Such an inequality has been extended in \cite{chen1971theorem} to submanifolds of any co-dimension in $\mathbb R^n$ for $n\geq3$. In particular, for a bounded domain $\Omega$ in $\mathbb R^n$ with
smooth boundary there holds
\begin{equation*}
\int_{\partial\Omega}|\frac{H}{n-1}|^{n-1}\geq|\mathbb S^{n-1}|,
\end{equation*}
with equality attained if and only if $\Omega$ is a ball.

In 2020, Agostiniani-Fogagnolo-Mazzieri \cite{agostiniani2020sharp} used a powerful and elegant
approach based on nonlinear potential theory  to  establish Willmore equality  for a bounded domain $\Omega$ with smooth boundary in complete Riemannian manifolds with non-negative Ricci curvature
\begin{equation}\label{Willmore}
\int_{\partial\Omega}|\frac{H}{n-1}|^{n-1}\geq{\rm AVR}(g)|\mathbb S^{n-1}|,
\end{equation}
where {\rm AVR}(g) is  asymptotic volume ratio defined by
$$
 {\rm AVR}(g)=\lim\limits_{r\rightarrow\infty}\frac{n\cdot {\rm Vol}\ B(p,r)}{r^n\cdot {\rm Vol}\ \mathbb S^{n-1}}.
$$
Whereafter, Wang \cite{wang2021remark}  provides a standard comparison methods in Riemannian geometry to prove Willmore equality \eqref{Willmore} which is very beautiful and concise.

Inspired by these previous results, we can inquire as to whether exist similar Willmore equality in complete Riemannian manifolds with negative Ricci curvature below bounded.

In this paper, we employ the methodology presented in \cite{wang2021remark} and offer a positive answer as following.  We need to present a few notions in order to effectively articulate our theorem.
\begin{definition}
Suppose that $(M,g)$ is an $n+1$-dimensional complete  Riemannian manifolds with bounded Ricci curvature from below:
$Ric\geq -ng.$
Let $\Omega$ be a bounded open subset of $M$ with smooth boundary, then we call
\begin{equation}\label{eqn:1.2}
  {\rm RV}(\Omega)=\lim\limits_{R\rightarrow+\infty} \frac{{\rm Vol}\{x\in M:d_g(x,\Omega)\leq R\}}{\omega_n\int_0^R \sinh^nsds}
\end{equation}
the relative volume of $\Omega.$ Here $\omega_n$ is the volume of the unit sphere $\mathbb{S}^n.$
\end{definition}

\begin{remark}\label{rem:1.2}
\begin{itemize}
\item (1)  ${\rm RV}(\Omega)$ is well defined (see Lemma \ref{lem:2.3}.)
\item (2)  The formula \eqref{eqn:1.2} can also be written as:
\begin{equation}\label{1.3}
\begin{aligned}
{\rm RV}(\Omega)&=\frac{n\cdot2^n}{\omega_n}\lim\limits_{R\rightarrow+\infty} \frac{{\rm Vol}\{x\in M:d_g(x,\Omega)\leq R\}}{e^{nR}}
\\ &=\frac{2^n}{\omega_n}\lim\limits_{R\rightarrow+\infty} \frac{{\rm Vol}\{x\in M:d_g(x,\Omega)= R\}}{e^{nR}}
\end{aligned}
\end{equation}
\item (3)  It is easy to see by definition that If $\Omega_1\subseteq \Omega_2,$ then ${\rm RV}(\Omega_1)\leq {\rm RV}(\Omega_2).$
\item (4) If $\Omega=B(p,r)$ is a geodesic ball, then
\begin{equation}
  {\rm RV}(B(p,r))=e^{nr}\cdot {\rm RV}(p)
\end{equation}
where $${\rm RV}(p)=\lim\limits_{R\rightarrow+\infty} \frac{{\rm Vol}\ B(p,R) }{\omega_n\int_0^R \sinh^nsds}$$ is well-defined by Bishop-Gromov volume comparison theorem.
\end{itemize}
\end{remark}

Here is the first main result in this paper:
\begin{theorem}\label{thm:1.3}
Let $(M,g)$ be a complete noncompact Riemannian manifold of $n+1$-dimension $(n\geq 2)$ with $Ric\geq -ng$ and $\Omega\subseteq M$ is a bounded open subset with smooth boundary $\partial\Omega.$ Then
\begin{equation}\label{1.5}
  \int_{\partial\Omega\bigcap\{H\geq-n\}}(1+\frac{H}{n})^nd\sigma\geq {\rm RV}(\Omega)\cdot\omega_n.
\end{equation}
where $H$ is the mean curvature of $\partial\Omega$ with respect to the outer normal. Moreover, if $H>-n$ and $K=(\frac{{\rm RV}(\Omega)\cdot\omega_n}{V(\partial\Omega)})^{\frac{1}{n}}-1>-1,$ then the equality holds if and only if one of the following cases occurs:
\begin{itemize}
  \item (1) If $H(p)>n$ for some $p\in\partial\Omega,$ then $\partial\Omega$ is connected, $H_{\partial\Omega}\equiv nK>n$ and $M\setminus\Omega$ is isometric to
    $$([{\rm arccoth}\ K,+\infty)\times \partial\Omega,dr^2+(K^2-1)\sinh^2r g_{\partial\Omega}).$$
    \item (2) If $H(p)=n$ for some $p\in\partial\Omega,$ then $\partial\Omega$ is connected, $H_{\partial\Omega}\equiv nK=n$ and $M\setminus\Omega$ is isometric to
        $$([1,+\infty)\times \partial\Omega,dr^2+e^{2r-2}g_{\partial\Omega}).$$
    \item (3) If $H(p)\in(-n,n)$ for some $p\in\partial\Omega,$ then $H_{\partial\Omega}\in(-n,n),$ $\partial\Omega=\bigcup\limits_{j=1}^{N}\Sigma_j$ where $\Sigma_j$ is the connected component of $\partial\Omega$
    and $N$ is the number of ends of $M$ while there exists some constants $s_1,s_2,\cdots,s_{N}$ such that
    $M\setminus\Omega$ is isometric to
    $$\bigcup\limits_{j=1}^{N}([s_j,+\infty)\times \Sigma_j, dr^2+\frac{\cosh^2r}{\cosh^2s_j}g_{\Sigma_j})$$
   In special, if $\partial\Omega$ is connected, then $H_{\partial\Omega}\equiv nK\in(-n,n)$ and  $M\setminus\Omega$ is isometric to
    $$([{\rm arctanh}\ K,+\infty)\times \partial\Omega,dr^2+(1-K^2)\cosh^2r g_{\partial\Omega}).$$

\end{itemize}
\end{theorem}
\begin{remark}
\begin{itemize}
\item [(1)]By Theorem C in \cite{kasue1983ricci} or Proposition 2 in \cite{cai1999boundaries}, we know that if $M\setminus\Omega$ is noncompact, then $\partial\Omega\cap \{H\geq -n\}\not=\emptyset$.
\item [(2)] We provide example 5.1 and example 5.2 to illustrate why we need to add condition '$H>-n$' to characterize the equality of (\ref{1.5}).
\end{itemize}
\end{remark}
We shall demonstrate how the definition of asymptotically hyperbolic manifold and the rigidity result are connected.
Let's review some conceptions of AH manifolds.

Let $\overline{M}$ be a compact smooth   manifold of dimension $n+1$ with non-empty
boundary $\partial M$ of dimension $n$ and $M$ be its interior. We say that a complete metric $g$ on $M$ is (smoothly)
conformally compact if there exits a defining function $\rho$ on $M$ such that the conformally
equivalent metric $\bar{g} = \rho^2g$  can extend to a smooth Riemannian metric on $\overline{M}.$
The defining function is smooth on $\overline{M}$ and satisfies
 \begin{equation}
    \rho>0\ \ {\rm in} \ M,\ \ \ \ \ \ \ \ \rho=0\ \ {\rm on}\ \partial M,\ \ \ \ \ \ \ d\rho\neq 0\ \ {\rm on}\  \partial M.
 \end{equation}
\par Suppose that $(M,g)$ is a conformally compact Riemannian manifold and $\bar{g}=\rho^2g$ is a compactification, we call the induced metric $\hat{g}=\bar{g}|_{\partial M}$ the boundary metric associated to the compactification $\bar{g}.$ Then $(M,g)$ carries a well-defined conformal structure on the boundary $\partial M$  and we call $(\partial M,[\hat{g}])$ the conformal infinity of $g.$
 \par A conformally compact Riemannian manifold is said to be asymptotically hyperbolic (AH) if its sectional curvature goes to $-1$ when
approaching the boundary at infinity $(\rho\rightarrow 0).$
\\
\par Let's look back on the three cases in Theorem 1.3. if we set $\rho=e^{-r}$ for largely $r,$ then $\rho^2g=d\rho^2+f^2(\rho)g_\Omega$ where $\rho\in(0,c)$ for some $c\in(0,1)$ and
$$f(\rho)=(K^2-1)\frac{1-\rho^2}{2},\ \ \ \ \ \frac{1}{e^2},\ \ \ \ \ \frac{1+\rho^2}{2\cosh s_j}$$
in the three cases above. Then $f(\rho)$ could be extended to $\rho=0$ smoothly. Hence $M$ is conformally compact and $\rho$ is a geodesic defining function, i.e. $|d\rho|^2_{\rho^2g}=1$ near infinity. On the other hand,
a direct calculation indicates that $|K_g+1|=O(e^{-2r})=O(\rho^2).$  Hence we have the folllowing result:
\begin{corollary}
Let $(M,g)$ be a complete noncompact Riemannian manifold of $n+1$-dimension $(n\geq 2)$ with $Ric\geq -ng$ and $\Omega\subseteq M$ is a bounded open subset with smooth boundary $\partial\Omega.$ If the mean curvature $H>-n,$ and
$$ \int_{\partial\Omega}(1+\frac{H}{n})^nd\sigma= {\rm RV}(\Omega)\cdot\omega_n. $$
Then $M$ is an asymptotically hyperbolic manifold with conformal infinity $(\partial \Omega,[g_{\partial\Omega}]).$ Furthermore, if $H(p)\geq n$ for some $p\in\partial \Omega,$ then the conformal infinity is connected.
\end{corollary}

\par As the second part of the paper, we consider Willmore-type inequality in Hyperbolic $\mathbb H^{n+1}$.
Recall that Chen \cite{chen1974some} proved that for any closed $2$-surface $\Sigma^2\subset\mathbb H^{n+1}$
Willmore-type equality holds
\begin{equation*}
\frac14\int_{\Sigma}H^2\geq 4\pi-|\Sigma|
\end{equation*}
with equality attained if and only if $\Sigma$ is a geodesic sphere. For other similar results, one can see \cite{chai2020willmore}\cite{wang2015singularities}.
In 2018, Hu \cite{hu2018willmore} proved that any star-shaped and mean-convex hypersurface $\Sigma\subset\mathbb H^n\, (n\geq3)$
satisfies
$$
\int_{\Sigma}(H^2-1)d\sigma\geq\omega_{n-1}^{\frac1{n-1}}|\Sigma|^{\frac{n-3}{n-1}}.
$$
Equality  holds if and only if $\Sigma$ is a geodesic sphere.

\par As an application of Theorem 1.3, we could provide a new Willmore-type inequality and it has very few requirements for surfaces. First, we define the inscribed radius as follows:
\begin{definition}
Let $\Omega$ be a bounded set in Riemannian manifold  $(M,g),$ we set
$$I(\Omega)=\sup\limits_{x\in \Omega}d_g(x,\partial\Omega)$$
to be the inscribed radius of $\Omega.$
\end{definition}
By the completeness of $\overline{\Omega},$ we can find some $p\in\Omega$ such that $B(p,I(\Omega))\subseteq \Omega.$
Then under the conditions of Theorem 1.3,
\begin{equation}\label{1.7}
  \begin{aligned}
  \int_{\partial\Omega\cap\{H\geq-n\}}(1+\frac{H}{n})^nd\sigma & \geq {\rm RV}(\Omega)\cdot\omega_n
  \\ &\geq {\rm RV}(B(p,I(\Omega))\cdot\omega_n
  \\ &=e^{nI(\Omega)}\cdot\omega_n\cdot {\rm RV}(p).
\end{aligned}
\end{equation}
for some $p\in \Omega.$
\par  Notice that ${\rm RV}(p)=1$ in hyperbolic space for any $p\in \mathbb H^{n+1},$ then we get the following Willmore-type inequality in hyperbolic space:
\begin{theorem}\label{thm:1.7}
  Let $\Omega$ be a bounded open subset in the $n+1$-dimensional hyperbolic space $\mathbb{H}^{n+1}$ with smooth boundary $\partial\Omega.$
  Then
\begin{equation}\label{1.8}
 \int_{\partial\Omega\cap\{H\geq-n\}}(1+\frac{H}{n})^nd\sigma\geq e^{nI(\Omega)}\cdot\omega_n
 \end{equation}
  where $H$ is the mean curvature of $\partial\Omega$ with respect to outer normal. Moreover the equality holds if and only if $\Omega$ is a geodesic ball in $\mathbb{H}^{n+1}.$
\end{theorem}

\begin{remark}
\begin{itemize}
\item (1) Let $\Omega\subseteq\mathbb{H}^{n+1}$ be a bounded open set with smooth convex boundary, i.e. all the principal curvatures of $\partial\Omega$ are nonnegative everywhere.  Then the volume of level set can be accurately given by Steiner’s formula \cite{allendoerfer1948steiner}:
$$
 {\rm Vol}\{x\in M:d_g(x,\Omega)=t\}= \sum_{k=0}^{n}\sinh^{n-k} t\cosh^k t \int_{\partial\Omega}H_k d\sigma.
$$
where $H_k=\sum\limits_{i_1<i_2<\cdots<i_k}\kappa_{i_1}\kappa_{i_2}\cdots\kappa_{i_k}$ is the $k-$th order
mean curvature and $(\kappa_{1},\kappa_{2},\cdots,\kappa_{n})$ are the eigenvalues of the second fundamental form of $\partial\Omega.$ Then by (\ref{1.3}) and MacLaurin inequality,
$$
{\rm RV}(\Omega)\cdot\omega_n =\sum_{k=0}^{n}\int_{\partial\Omega}H_kd\sigma \leq \sum_{k=0}^{n}\int_{\partial\Omega} C_n^k (\frac{H}{n})^k d\sigma
=\int_{\partial\Omega}(1+\frac{H}{n})^nd\sigma.
$$

\item (2) Let $(\Omega, g)$ be an $n+1$-dimensional complete Riemannian manifold with
nonempty boundary $\partial \Omega$. If $Ric\geq-n$ and $H_{\partial\Omega}\geq nk>n$.  Then we have the estimate of the inscribed radius:
$$
I(\Omega)=\sup\limits_{x\in \Omega}d_g(x,\partial \Omega)\leq {\rm arccoth} k.
$$
The equality holds if and only if $(\Omega, g)$ is isometric
to a geodesic ball of radius ${\rm arccoth} k$ in hyperbolic space. One can see \cite{kasue1983ricci} or \cite{li2015rigidity} for more details.
\par The formula (\ref{1.8}) could also provide an estimate of the inscribed radius in the hyperbolic space.
\end{itemize}
\end{remark}

Here is the outline of this paper: In section 2, we will provide some basic analysis in Riemannian geometry and ODEs to prove that the relative volume is well-defined and to make preparations for the proof of Theorem 1.3. Next, we prove Theorem 1.3 in section 3. The primary method is standard and resembles that in \cite{wang2021remark}. The rigidity result of the Willmore-type inequality in hyperbolic space is proved in section 4. We make use of the property of conformally compactification of hyperbolic space to obtain the result. In the end, We provide two examples to illustrate that the condition '$H>-n$' in Theorem 1.3 is necessary.

\section{A preliminary analysis in Riemannian geometry}
Let $\Omega\subseteq M$ be the open bounded set with smooth boundary $\partial\Omega$ and $H$ be the mean curvature. For any $p\in \partial\Omega,$ we set $\sigma_p:[0,T)\rightarrow M\setminus\Omega$ to be the normal geodesic satisfying $\sigma_p(0)=p$ and $\sigma_p'(0)\bot T_p\partial\Omega.$
We define the function $\tau:\partial\Omega\rightarrow \mathbb{R}\bigcup\{+\infty\}$ by
\begin{equation}
\tau(p)=\sup\{t>0:d_{g}(\sigma_p(t),\partial\Omega)=t\}.
\end{equation}
 We call $Cut(\partial\Omega)=\{\sigma_p(\tau(p)):p\in \partial\Omega\}$ the cut
locus for the boundary $\Sigma=\partial\Omega$ in $M\setminus\Omega.$ The classic Riemannian geometry theory tells us that $Cut(\partial\Omega)$ is a closed subset of measure zero and the distance function $d_g(\partial\Omega, \cdot)$ is smooth in $(M\setminus\Omega)\setminus Cut(\partial\Omega).$ Furthermore, if $q = \sigma_p(\tau(p))\in Cut(\partial\Omega),$  then either $q$ is the first focal point of $\partial\Omega$ along $\sigma_p,$ or
there exists another focal point $p'$ on $\partial\Omega$ such that $q = \sigma_{p'}(\tau(p')).$
\par Let $E=\{(p,r)\in\partial\Omega\times[0,+\infty):r<(\tau(p))\}$ and we have the diffeomorphism
\begin{equation}
\Phi:E\rightarrow (M\setminus\Omega)\setminus Cut(\partial\Omega),\ \ \ (p,r)\mapsto \sigma_p(r).
\end{equation}
We use $d\mu=J(p,r)drd\sigma(p)$ denote the volume form and $H(p,r)$ to denote the mean curvature of the level set of the distance function to $\partial\Omega$ in $M\setminus\Omega.$ For any fixed $p\in\partial\Omega,$ we have the following inequality by the Bochner formula and Riccati equation:
\begin{equation*}
\begin{aligned}
 0=\frac12\Delta|\nabla r|^2=&|D^2r|^2+\langle \nabla r,\,\nabla \Delta r\rangle+Ric(\nabla r,\nabla r) \\
 \geq&\frac{(\Delta r)^2}{n}+\langle \nabla r,\,\nabla \Delta r\rangle-n,
  \end{aligned}
\end{equation*}
that is
\begin{equation}
H'(r)+\frac{H^2(r)}{n}\leq n, \ \ \ \ \ r\in [0,\tau(p)),
\end{equation}
where $H(r)$ is mean curvature of level set  and $\Delta r|_{r=0}=H_p(0)=H(p)$.

\par Let's study the solution $f_p(r)$ to the following ODE:
\begin{equation}
  f_p'(r)+f_p^2(r)=1,\ \ \ \ \ \ f_p(0)=\frac{H(p)}{n} , \ \ r\geq 0.
\end{equation}
 It is easy to find that:

\begin{itemize}
\item [(1)] If $H(p)=\pm n,$ then $f_p(r)=\pm 1.$

\item [(2)] If $|H(p)|<n,$ then $f_p(r)=\tanh(r+r_1)$ where $r_1={\rm arctanh} \frac{H(p)}{n}.$

\item [(3)] If $|H(p)|>n,$ then $f_p(r)=\coth(r+r_2)$ where $r_2={\rm arccoth} \frac{H(p)}{n}.$
Moreover, if $H(p)<-n,$ then $r\in [0,-r_2).$
\end{itemize}
Furthermore, we have
\begin{equation}
\frac{J'(p,r)}{J(p,r)}=H(r)\leq nf_p(r).
\end{equation}
By the above facts, we could obtain the following Lemma.
\begin{lemma}\label{lem:2.1}
For any $r<\tau(p),$ we have that
$$
 \frac{J(p,r)} {(\cosh r+\frac{H(p)}{n}\sinh r)^n}
$$
is non-increasing. Moreover,
\begin{equation}\label{2.6}
 J(p,r)\leq {(\cosh r+\frac{H(p)}{n}\sinh r)^n}.
\end{equation}
\end{lemma}
\begin{proof}
We will divide it into three cases:

{\bf Case 1}: $H(p)=\pm n$;
{\bf Case 2}: $|H(p)|< n$;
{\bf Case 3}: $|H(p)|> n$.

\vskip 2mm
For the {\bf Case 1} $H(p)=\pm n$, we see that $f_p(r)=\pm1$. Moreover $\frac{J'}{J}\leq\pm n.$ It is easy to see that $\frac{J(p,r)}{e^{\pm nr}}$ is non-increasing for
$r<\tau(p)$ as $(\frac{J(p,r)}{e^{\pm nr}})'\leq 0.$ Therefore $\frac{J(p,r)}{e^{\pm nr}}\leq \frac{J(p,0)}{e^{0}}=1$ which would implies (\ref{2.6}).

\vskip 2mm
For the {\bf Case 2} $|H(p)|< n$, we have $f_p(r)=\tanh(r+r_1)$. Moreover
$$
J'\leq nJ\frac{\tanh r+\frac{H(p)}{n}}{1+\frac{H(p)}{n}\tanh r}.
$$
Then we can see that
$$
\frac{J}{(\cosh r+\frac{H(p)}{n}\sinh r)^n}
$$
is non-increasing. Therefore $ J(p,r)\leq {(\cosh r+\frac{H(p)}{n}\sinh r)^n}  $.

\vskip 2mm
For the {\bf Case 3} $|H(p)|> n$, we know that $f_p(r)=\coth(r+r_2)$. Moreover
$$
J'\leq nJ\frac{\coth r+\frac{H(p)}{n}}{1+\frac{H(p)}{n}\coth r}.
$$
Then we can see that
$$
\frac{J}{(\cosh r+\frac{H(p)}{n}\sinh r)^n}
$$
is non-increasing. Therefore $ J(p,r)\leq {(\cosh r+\frac{H(p)}{n}\sinh r)^n} $.

This completes the proof of Lemma \ref{lem:2.1}.
\end{proof}


\begin{lemma}\label{lem:2.2}
  Set $\bar{J}(p,r)=J(p,r)$ for $r<\tau(p)$ and $\bar{J}(p,r)=0$ for $r\geq \tau(p).$ Then
  $\lim\limits_{r\rightarrow+\infty}\frac{\bar{J}(p,r)}{e^{nr}}$ exists.
\end{lemma}

\begin{proof}
For $r\geq \tau(p)$, it is obvious.  We only need to consider the  $r<\tau(p)$. On the one hand, from Lemma 2.1, we can see that
$\tfrac{J}{(\cosh r+\frac{H(p)}{n}\sinh r)^n}$ is non-increasing and bounded which implies that
$\lim\limits_{r\rightarrow+\infty}\tfrac{J}{(\cosh r+\frac{H(p)}{n}\sinh r)^n}$ exists.
On the other hand,
$\lim\limits_{r\rightarrow+\infty}\tfrac{(\cosh r+\frac{H(p)}{n}\sinh r)^n}{e^{nr}}$ exists.

Thus, we complete the proof of Lemma \ref{lem:2.2}.

\end{proof}
\begin{lemma}\label{lem:2.3}
Suppose that $(M,g)$ is an $n+1$-dimensional complete Riemannian manifold with
$Ric\geq -ng.$
Let $\Omega$ be a bounded open subset of $M$ with smooth boundary. Then $$\lim\limits_{R\rightarrow+\infty} \frac{{\rm Vol}\{x\in M:d_g(x,\Omega)\leq R\}}{e^{nR}}$$
 exists. Hence the relative volume
${\rm RV}(\Omega)$ is well-defined.
\end{lemma}
\begin{proof}
It follows from  Lemma \ref{lem:2.2} that $\lim\limits_{R\rightarrow+\infty}\frac{\int_0^R\bar{J}(p,r)dr}{e^{nR}}=j(p)$ and  there exists a number $C\in\mathbb R_{+}$ such that $\bar J(p,r)\leq Ce^{nr}$. Then
\begin{equation*}
\frac{\int_0^R\bar{J}(p,r)dr}{e^{nR}}\leq \frac{\int_0^R Ce^{nr}dr}{e^{nR}}\leq C_1.
\end{equation*}
By  Lebesgue dominated convergence theorem, we have that
\begin{equation}
\begin{aligned}
\lim\limits_{R\rightarrow+\infty} \frac{{\rm Vol}\{x\in M:d_g(x,\Omega)\leq R\}}{e^{nR}}=&\lim\limits_{R\rightarrow+\infty}\frac{\int_{\partial\Omega}\int_0^R\bar{J}(p,r)drd\sigma(p)}{e^{nR}}\\
=&\int_{\partial\Omega}j(p)d\sigma(p).
\end{aligned}
\end{equation}
This completes the proof Lemma \ref{lem:2.3}.
\end{proof}
\section{Proof of the main theorem}
Under the conditions of Theorem 1.3, we set $\Sigma_1=\partial\Omega\bigcap\{H\geq -n\}$ and $\Sigma_2=\partial\Omega\bigcap\{H< -n\}.$
Then
\begin{equation}
\lim\limits_{R\rightarrow+\infty}\frac{1}{R}\int_{\Sigma_2}\int_0^R \bar{J}(p,r)drd\sigma(p)=0
\end{equation}
as $\bar{J}(p,r)=0$ for $r$ large enough. Hence $\int_{\Sigma_2}\int_0^R \bar{J}(p,r)drd\sigma(p)=o(R).$
 For $R>0$ big enough,
\begin{equation}\label{eqn:3.1}
  \begin{aligned}
    {\rm Vol}&(x\in M: d_g(x,\Omega)\leq R)-{\rm Vol}(\Omega)=\int_{\partial\Omega}\int_0^R\bar{J}(p,r)drd\sigma(p)
    \\ &\leq \int_{\Sigma_1}\int_0^R(\cosh r+\frac{H(p)}{n}\sinh r)^ndrd\sigma(p)+ \int_{\Sigma_2}\int_0^R  \bar{J}(p,r)drd\sigma(p)
    \\ &=\int_{\Sigma_1}\int_0^R[1+\frac{H(p)}{n}+(1-\frac{H(p)}{n})e^{-2r}]^n\cdot (\frac{e^r}{2})^n drd\sigma(p)+o(R)
    \\ &=\int_{\Sigma_1}\int_0^R[(1+\frac{H(p)}{n})^n(\frac{e^r}{2})^n+O(e^{(n-2)r})] drd\sigma(p)+o(R)
    \\ &= \frac{e^{nR}}{n\cdot2^n}\int_{\Sigma_1}(1+\frac{H(p)}{n})^n d\sigma(p)+O(e^{(n-2)R})+o(R)
  \end{aligned}
\end{equation}
Multiply both sides by $\frac{n\cdot2^n}{e^{nR}}$ and let $R\rightarrow+\infty$, we have
\begin{equation*}
 \int_{\partial\Omega\bigcap\{H\geq-n\}}(1+\frac{H(p)}{n})^nd\sigma\geq {\rm RV}(\Omega)\cdot\omega_n.
\end{equation*}

\subsection{Rigidity}
We are going to prove the rigidity result in Theorem 1.3. The ``if" part is easy from a direct computation. Here we only provide the procedure of case three, i.e.  if $M\setminus\Omega$ is isometric to
    $$\bigcup\limits_{j=1}^{N}([s_j,+\infty)\times \Sigma_j, dr^2+\frac{\cosh^2r}{\cosh^2s_j}g_{\Sigma_j})$$
for some constants $\{s_j\}_{j=1}^N,$ then $\int_{\partial\Omega}(1+\frac{H(p)}{n})^nd\sigma= {\rm RV}(\Omega)\cdot\omega_n.$ Notice that
the mean curvature $H$ on $\Sigma_j$ is just $n\tanh s_j\in(-n,n).$ Then
\begin{equation}
  \begin{aligned}
  \int_{\partial\Omega}(1+\frac{H(p)}{n})^nd\sigma&=\sum\limits_{j=1}^N \int_{\Sigma_j}(1+\frac{H(p)}{n})^nd\sigma=\sum\limits_{j=1}^N \int_{\Sigma_j}(1+\tanh s_j)^nd\sigma
  \\&=\sum\limits_{j=1}^N (\frac{e^{s_j}}{\cosh s_j})^n\cdot {\rm Vol}(\Sigma_j)
  \end{aligned}
\end{equation}
On the other hand,
\begin{equation}
  \begin{aligned}
  {\rm RV}(\Omega)\cdot \omega_n&=n\cdot 2^n\lim\limits_{R\rightarrow+\infty} \frac{{\rm Vol}\{x\in M:d_g(x,\Omega)\leq R\}}{e^{nR}}
  \\&=n\cdot2^n\sum\limits_{j=1}^N \lim\limits_{R\rightarrow+\infty}\frac{1}{e^{nR}}\int_{\Sigma_j}\int_{s_j}^{s_j+R}\frac{\cosh^n r}{\cosh^n s_j}drd\sigma
  \\&=2^n\sum\limits_{j=1}^N\lim\limits_{R\rightarrow+\infty}\int_{\Sigma_j}\frac{1}{e^{nR}}\frac{\cosh^n(s_j+R)}{\cosh^n s_j}d\sigma
  \\&=\sum\limits_{j=1}^N (\frac{e^{s_j}}{\cosh s_j})^n\cdot {\rm Vol}(\Sigma_j).
  \end{aligned}
\end{equation}
Hence the equality holds.
\par In the following, we prove the ``only if" part. Firstly, we have that
\begin{lemma}
Let $M,\Omega$ be defined as in Theorem 1.3, and suppose that $H_{\partial\Omega}> -n$ and
\begin{equation}\label{3.5}
 \int_{\partial\Omega}(1+\frac{H}{n})^nd\sigma= {\rm RV}(\Omega)\cdot\omega_n.
\end{equation}
Then on $\partial\Omega,$ we have that  $\tau=\infty$ and $J(p,r)= {(\cosh r+\frac{H(p)}{n}\sinh r)^n}$ for all $r>0.$
\end{lemma}
\begin{proof}
Let $\Sigma_3=\partial\Omega\cap\{\tau=\infty\}$ and $\Sigma_4=\partial\Omega\cap\{\tau<\infty\}.$ Set
$$Q(p,r)=\frac{J(p,r)}{(\cosh r+\frac{H(p)}{n}\sinh r)^n}$$
which is non-increasing. Then for any fixed $R' < R$ we have
\begin{equation*}
  \begin{aligned}
    &\ \ \ \ \ \ {\rm Vol}(x\in M: d_g(x,\Omega)\leq R)-{\rm Vol}(\Omega)
    \\ &=\int_{\Sigma_3}\int_0^R\bar{J}(p,r)drd\sigma(p)+\int_{\Sigma_4}\int_0^R\bar{J}(p,r)drd\sigma(p)
    \\ &=\int_{\Sigma_3}\int_0^R J(p,r)drd\sigma(p)+o(R)
    \\ &= \int_{\Sigma_3}\int_0^{R'}J(p,r)drd\sigma(p)+ \int_{\Sigma_3}\int_{R'}^{R}Q(p,r)(\cosh r+\frac{H(p)}{n}\sinh r)^ndrd\sigma(p)
    +o(R)\\
     &\leq O(1)+ \int_{\Sigma_3}Q(p,R')\int_{R'}^{R}(\cosh r+\frac{H(p)}{n}\sinh r)^ndrd\sigma(p)
   +o(R)\\
   &=O(1)+\frac{e^{nR}}{n\cdot2^n}\int_{\Sigma_3}Q(p,R')(1+\frac{H(p)}{n})^n d\sigma(p)+O(e^{(n-2)R})+o(R)
  \end{aligned}
\end{equation*}
Multiply both sides by $\frac{n\cdot2^n}{e^{nR}}$ and let $R\rightarrow+\infty$, we have
\begin{equation*}
 {\rm RV}(\Omega)\cdot\omega_n\leq \int_{\Sigma_3}Q(p,R')(1+\frac{H(p)}{n})^nd\sigma.
\end{equation*}
Let $R'\rightarrow+\infty,$ we obtain that
\begin{equation*}
{\rm RV}(\Omega)\cdot\omega_n\leq \int_{\Sigma_3}Q_{\infty}(p)(1+\frac{H(p)}{n})^nd\sigma,
\end{equation*}
where $Q_{\infty}(p)=\lim\limits_{r\rightarrow+\infty}Q(p,r)\leq1$.  As we have equality  (\ref{3.5}) and $H>-n,$ we must have $Q_{\infty}(p)=1$ for a.e. $p \in\Sigma_3$ and $\partial\Omega\setminus\Sigma_3$ is zero measure set. It follows that
$$
 J(p,r)=(\cosh r+\frac{H(p)}{n}\sinh r)^n \ \ \ {\rm on}\ [0,\infty)
$$
for a.e. $p \in\partial\Omega$. By continuity the above identity holds for all $p\in\partial\Omega$.
\end{proof}
Lemma 3.1 implies that if the equality (\ref{3.5}) holds, we must have that on $\Phi([0,\infty)\times\partial\Omega)$
\begin{equation}
D^2r=\frac{H(p,r)}{n}g=f_p(r)g,\ \ Ric(\nabla r,\nabla r)=-n.
\end{equation}
As $Ric\geq-ng$, it follows that $Ric(\nabla r,X)=0$ for  $X\bot\nabla r.$ From the 1st equation above, we know
$\partial\Omega$ is an umbilical hypersurface, i.e. the second fundamental form $h=f_p(r)g_{\partial\Omega}$. Let $\{e_0=\nu,e_1,\cdots,e_n\}$
be orthonormal frmae along $\partial\Omega$. By the Codazzi equation, with $1\leq\, i,j,k\leq n$, we have
\begin{equation}
R(e_k,e_j,e_i,\nu)=h_{ij,k}-h_{ik,j}=\frac1n(H_k\delta_{ij}-H_j\delta_{ik}).
\end{equation}
It follows that
\begin{equation*}
-\frac{n-1}nH_j=Ric(e_j,\nu)=0.
\end{equation*}
Thus $H$ is locally constant  on $\partial\Omega$. Therefore $\partial\Omega$ must be the union of several
components of $\partial\Omega.$ Let $\Sigma\subseteq\partial\Omega$ be a connected component such that the mean curvature is a constant $nK$ on $\Sigma.$  We know that $\Phi $ is a diffeomorphism from $[0,\infty)\times\Sigma$
 onto its image and the pullback metric $\Phi^{\ast}g$ takes the following form
 $$
 dr^2+{g}_r,
 $$
where ${g}_r$ is a $r$-dependent family of metrics on $\Sigma$ and $ g_0=g_{\Sigma}$.
Since
\begin{equation}
D^2r=
\left\{
\begin{aligned}
& g\ \ \ \ \ \ \ \ \ \ \ \ \ \ \ \ \ \ \ \ \ \ \ \ \ \ \ \ \ \ \ \ \ \ \ \ \ \ \ \ \  \ \  \text{if} \ K=1,\\
&\tanh(r+{\rm arctanh}K)\,g\ \ \ \ \ \ \ \text{if} \ |K|<1,\\
&\coth(r+{\rm arccoth}K)\,g\ \ \ \ \ \ \ \ \text{if} \ K>1.
\end{aligned}\right.
\end{equation}
In the local coordinates $\{x_1,x_2,\cdots,x_n\}$ on $\Sigma$, we have
\begin{equation}
\frac12\frac{\partial}{\partial r} g_{ij}=\left\{
\begin{aligned}
& \, g_{ij}\ \ \ \ \ \ \ \ \ \ \ \ \ \ \ \ \ \ \ \ \ \ \ \ \ \ \ \ \ \ \ \ \ \ \ \  \ \ \  \ \ \ \text{if} \ K=1,\\
&\tanh(r+{\rm arctanh}K)\,g_{ij}\ \ \ \ \ \ \ \text{if} \ |K|<1,\\
&\coth(r+{\rm arccoth}K)\,g_{ij}\ \ \ \ \ \ \ \ \text{if} \ K>1.
\end{aligned}\right.
\end{equation}
Therefore
\begin{equation}\label{3.10}
 g_r=\left\{
\begin{aligned}
&e^{2r}\, g_{\Sigma}\ \ \ \ \ \ \ \ \ \ \ \ \ \ \ \ \ \ \ \ \ \ \ \  \ \ \ \ \ \ \ \ \ \ \ \ \ \ \ \ \ \ \ \ \ \ \ \ \  \ \ \ \text{if} \ K=1,\\
&(1-K^2)\cosh^2(r+{\rm arctanh}K)g_{\Sigma}\ \ \ \  \ \ \text{if} \ |K|<1,\\
&(K^2-1) \sinh^2(r+{\rm arccoth}K) g_{\Sigma}\ \ \ \   \ \ \text{if} \ K>1.
\end{aligned}\right.
\end{equation}
\par Suppose that $\partial\Omega$ has $N$ connected components and each one has constant mean curvature, then the number of ends of $M$ (denoted by $N(M)$) is also $N$ as $(M\setminus\Omega,g)$ are some warped products depending on $\partial\Omega.$ Here
$$N(M)=\lim\limits_{k\rightarrow\infty}N_{\Omega_k}(M)$$ and
$N_{\Omega_k}(M)$ is the number of unbounded connected components of $M\setminus\Omega_k$ where $\{\Omega_k\}$ is
a compact exhaustion of $M.$
\\~\\
\textbf{Case (1):} $H(p)>n$ for some $p\in\partial\Omega.$
\par We still use $\Sigma$ to denote the connected component containing $p.$ Then $K>1$ on $\Sigma.$ The only thing we need to check is that $M$ has only one end. Here we utilize Theorem 3 in \cite{cai1999boundaries} to prove it. Set $\Sigma_k=\{k\}\times\Sigma$ for $k=1,2,\cdots.$ Then every $\Sigma_k$ is compact and separates $M.$ We also have that $d(p,\Sigma_k)\rightarrow+\infty$ as $k$ tends to infinity. Notice that $H|_{\Sigma_k}>n,$ then by Theorem 3 in \cite{cai1999boundaries}, we could conclude that
$M$ has one end or $(M,g)$ isometric to $(\mathbb{R}\times\Sigma,  dr^2 + e^{2r}g_0).$ By the formula of line 3 in (\ref{3.10}), $M$ has one end.
\\~\\
\textbf{Case (2):} $H(p)=n$ for some $p\in\partial\Omega.$
\par We construct $\Sigma$ and $\Sigma_k$ as above and the proof is the same as case (1). In this case $H|_{\Sigma_k}=n.$ By Theorem 3 in \cite{cai1999boundaries}, we could conclude that $M$ has either one end or $(M,g)$ isometric to $(\mathbb{R}\times\Sigma,  dr^2 + e^{2r}g_0).$ We claim that if we require that $H_{\partial\Omega}>-n,$ then the second scenario would not happen. Or else $\partial\Omega$ has another connected component $\Sigma'$ besides $\Sigma$ and the mean curvature of is also constant $c.$ In this case,
$M\setminus\Omega$ is isomorphic to $([0,+\infty)\times\Sigma,dr^2+e^{2r}g_\Sigma)\bigcup([0,+\infty)\times\Sigma',dr^2+e^{-2r}g_{\Sigma'}).$ Then
a direct calculation shows that
$$2^n{\rm Vol}(\Sigma)={\rm RV}(\Omega)\omega_n=\int_\Omega(1+\frac{H}{n})^nd\sigma=2^n{\rm Vol}(\Sigma)+(1+\frac{c}{n})^n{\rm Vol}(\Sigma').$$
See example 5.1 for more details. Then $c=-n$ and it contradicts the condition.
Then $M$ has only one end and hence $\partial\Omega$ is connected.
\\~\\
\textbf{Case (3):} $H(p)\in(-n,n)$ for some $p\in\partial\Omega.$
\par It is a direct conclusion of case (1) and (2) that $H_{\partial\Omega}\in(-n,n).$ Suppose that $N(M)=N$ and hence $\partial\Omega=\bigcup\limits_{j=1}^{N}\Sigma_j$ where $\Sigma_j$ is the connected component of $\partial\Omega$ and $H|_{\Sigma_j}=h_j\in(-n,n)$ is a constant. Set $s_j={\rm arctanh}\frac{h_j}{n}\in\mathbb{R},$ then case 3 in Theorem 1.3 is obtained by (\ref{3.10}).

\section{A Willmore typer inequality in hyperbolic space}
 In this section, we prove Theorem 1.7. The inequality (\ref{1.8}) is an easy corollary of Theorem 1.3 and the monotonicity of relative volume. In order to prove the rigidity result, we firstly recall that the relative volume is defined by the limit of the quotient in (1.2) and it is well-defined since the Ricci curvature is bounded from below. In the following, we introduce another method to calculate the relative volume on AH manifold of order 2 which was introduced in \cite{jin2022relative} and \cite{li2017gap}.
 \par Let $(M,g)$ be an $n+1$-dimensional conformally compact Riemannian manifold and its sectional curvature satisfies that $|K_g+1|=O(x^2)$ where $x$ is the geodesic defining function and $\bar{g}=x^2g$ is the geodesic compactification with boundary metric $\hat{g}=\bar{g}|_{\partial M}.$ Let $E\subset M$ be a compact subset or bounded open set, we use $s_E$ to denote the distance function of $E,$ i.e.
\begin{equation}
  \forall q\in M,\ \ s_E(q)=d_{g}(E,q)=\inf\limits_{y\in E}d_{g}(y,q).
\end{equation}
Let $r=-\ln x$ near the boundary and it is a distance function of $(M,g).$ We set the conformal change $g^E=e^{-2s_E}g$ on $M$  and could prove that $|d(s_E-r)|_{\bar{g}}$ is bounded  for sufficiently large $r$ almost everywhere (Lemma 4.1 in \cite{li2017gap}, Lemma
3.1 in \cite{dutta2010rigidity}).
Hence $g^E=e^{-2s_E}g=e^{2(r-s_E)}\bar{g}$ is Lipschitz continuous to the boundary.
 Then
\begin{equation}
\begin{aligned}
  {\rm RV}(E)&=\frac{2^n}{\omega_n}\lim\limits_{R\rightarrow+\infty} \frac{{\rm Vol}_g\{x\in M:d_g(x,E)= R\}}{e^{nR}}
\\ &=\frac{2^n}{\omega_n}\lim\limits_{R\rightarrow+\infty} {\rm Vol}_{g^E}\{x\in M:d_g(x,E)= R\}
\\ &=\frac{2^n}{\omega_n}\cdot {\rm Vol}(\partial M,\hat{g}^E)
\end{aligned}
\end{equation}
 where $\hat{g}^E=g^E|_{\partial X}.$

  \par Now we consider the $n+1$-dimentional hyperbolic space. Let $\Omega\subseteq\mathbb{H}^{n+1}$ be an open bounded set and $B(p,I(\Omega))\subseteq \Omega$ is the inscribed ball.
Set $s_p(\cdot) = d_g(p,\cdot)$ and $s_{\Omega}(\cdot)=d_g(\Omega,\cdot)$ be the distance functions and consider the conformal compact metric $g^p=e^{-2s_p}g$ and $g^\Omega=e^{-2s_\Omega}g$ with boundary metrics $\hat{g}^p$ and $\hat{g}^\Omega$ on the boundary $\mathbb{S}^n.$ We have that
$$1={\rm RV}(p)=\frac{2^n}{\omega_n}\cdot {\rm Vol}(\mathbb{S}^n,\hat{g}^p).$$
As it is showed above, $s_p-s_{\Omega}\geq I(\Omega)$ and it is Lipschitz continuous
to the boundary. Then $(s_p-s_{\Omega})|_{\mathbb S^n}\geq I(\Omega)$.
Hence
\begin{equation}\label{4.3}
 \begin{aligned}
   {\rm RV}(\Omega)&=\frac{2^n}{\omega_n}{\rm Vol}(\mathbb S^n,\hat g_{\Omega})=\frac{2^n}{\omega_n}\int_{\mathbb S^n}e^{n(s_{p}(q)-s_{\Omega}(q))}dV_{\hat g^p}
   \\ &\geq \frac{2^n}{\omega_n}\int_{\mathbb S^n} e^{nI(\Omega)}dV_{\hat g^p}=e^{nI(\Omega)}.
 \end{aligned}
\end{equation}
If the equality in (\ref{1.8}) holds, then equality ${\rm RV}(\Omega)=e^{nI(\Omega)}$ also holds and from (\ref{4.3}), we get that
\begin{equation}
(s_p-s_{\Omega})|_{\mathbb S^n}\equiv I(\Omega).
\end{equation}
by the continuous of $s_p-s_{\Omega}$ on $\mathbb{S}^n.$

Let $\sigma:[0,+\infty)\rightarrow\mathbb{H}^{n+1}$ be any geodesic ray from $p=\sigma(0)$ to $\hat{q}=\sigma(+\infty)\in\mathbb{S}^n.$
Let $q\in\partial\Omega\bigcap\sigma$ be any point when $\sigma$ meets $\partial\Omega.$ Then
$d(p,q)\geq I(\Omega)$ by the definition of inscribed radius. On the other hand,
\begin{equation}
   \begin{aligned}
     I(\Omega)&=(s_p-s_{\Omega})(\hat{q})=\lim\limits_{t\rightarrow+\infty}(d(p,\sigma(t)-d(\Omega,\sigma(t)))
     \\&=\lim\limits_{t\rightarrow+\infty}(d(p,q)+d(q,\sigma(t))-d(\Omega,\sigma(t)))
     \\&\geq \lim\limits_{t\rightarrow+\infty}d(p,q)=d(p,q).
   \end{aligned}
\end{equation}
Then $d(p,q)=I(\Omega)$ and hence $q$ is unique. We get that $q\in \partial B(p,I(\Omega)).$ Since the direction of $\sigma$ is arbitrarily selected, we finally obtain that
$\Omega=B(p,I(\Omega)).$

\section{Some examples}

\begin{example}
$(M ,g)=(\mathbb{R}\times N,dr^2+e^{2r}g_N)$ where $N$ is an $n$-dimensional compact manifold satisfying that $Ric_N\geq 0$. Then
$Ric_M\geq -ng.$ Set $\Omega=(T_1,T_2)\times N$ for some constant $T_1$ and $T_2.$
\end{example}
$\partial\Omega $ has two connected components: $\Sigma_1=\{T_1\}\times N$ and $\Sigma_2=\{T_2\}\times N.$ It is easy to calculate that the mean curvature  $H|_{\Sigma_1}=-n $ and $H|_{\Sigma_2}=n.$
Then
\begin{equation}
\begin{aligned}
{\rm RV}(\Omega)\omega_n&=n\cdot2^n\lim\limits_{R\rightarrow+\infty} e^{-nR}[V(\Omega)+\int_{\Sigma_2}\int_{0}^R e^{nr}drd\sigma+\int_{\Sigma_1}\int_{-R}^{0} e^{nr}drd\sigma] \\
&  =2^n\cdot {\rm Vol}(\Sigma_2)
  = \int_{\Sigma_1\cup\Sigma_2}(1+\frac{H}n)^nd\sigma
\\&=\int_{\partial\Omega}(1+\frac{H}n)^nd\sigma.
\end{aligned}
\end{equation}
In this case, $M$ has two ends and one of them is a cusp. Then $M$ dose not admit a conformal compactification. The condition $H>-n$ in Corollary 1.5 is in some sense sharp.

\begin{example}
 $(\mathbb R^{n+1},g)=(\mathbb R^{n+1},dr^2+\phi^2(r)g_{\mathbb S^n})$, where
\begin{equation}
\phi(r)=\left\{
\begin{aligned}
&3e^{r-2}\ \ \ \ \ \ \ \ \ \ r\geq2,\\
&r+\frac1{16}r^4\ \ \ r\in[0,2).
\end{aligned}\right.
\end{equation}
\end{example}
Then we have the following facts.
\begin{lemma}
The metric $g$ is $C^2$ and $Ric\geq-ng$.
\end{lemma}
\begin{proof}
It is trivial that $g\in C^2$ because
\begin{equation}
\left\{
\begin{aligned}
&\phi(0)=0, \ \ \phi'(0)=1,\ \ \phi''(0)=0,\\
&\phi(2)=\phi'(2)=\phi''(2)=3.
\end{aligned}\right.
\end{equation}

For $X\bot\nabla r$ and  $|X|=1$,
\begin{equation}
\left\{
\begin{aligned}
&Ric(X,X)=(n-1)\frac{1-\phi'^2}{\phi^2}-\frac{\phi''}{\phi},\\
&Ric(\nabla r,\nabla r)=-n\frac{\phi''}{\phi}.
\end{aligned}\right.
\end{equation}
Hence, we only need to check $\frac{\phi''}{\phi}\leq1$ and $\frac{\phi'^2-1}{\phi^2}\leq 1$.
If $r\geq 2$, it is easy to see that $\frac{\phi''}{\phi}=1$ and $\frac{\phi'^2-1}{\phi^2}\leq 1$.
If $r\in[0,2)$, we can see that $\frac34r^2\leq r+\frac1{16}r^4$ is obvious which implies that $\frac{\phi''}{\phi}\leq1$.
We compute that
\begin{equation}
\begin{aligned}
\phi^2-\phi'^2+1&=\frac{r^2}{256}(r^6-16r^4+32r^3-128r+256)\\
&=\frac{r^2}{256}\left((r^3-8r+8)^2+16(r^3-4r^2+12)\right)\geq0.
\end{aligned}
\end{equation}

This completes the proof of Lemma 5.3.
\end{proof}
With the presentations above, we could consider $\Omega=(3,4)\times\mathbb S^n\subseteq(\mathbb R^{n+1},g).$ Then $\partial\Omega=\Sigma_1\bigcup\Sigma_2=\{3\}\times \mathbb S^n \bigcup\{4\}\times \mathbb S^n$ and
$H|_{\Sigma_1}=-n$ and $H|_{\Sigma_2}=n$.
We find that $\tau(\Sigma_1)<\infty$ and it is a ``hole". Hence the condition $``H>-n"$ in Theorem 1.3 and Lemma 3.1 is necessary.
\\~\\
{\bf Acknowledgement.} The second author expresses gratitude to Professor Haizhong Li, Hui Ma and Daguang Chen for their invaluable assistance during his visit to Tsinghua University.
\bibliographystyle{plain}%

\bibliography{Jin-Yin2024}

\noindent{Xiaoshang Jin}\\
  School of mathematics and statistics, Huazhong University of science and technology, 430074, Wuhan, P.R. China
 \\Email address: jinxs@hust.edu.cn
 \\~\\
\noindent{Jiabin Yin}\\
 School of Mathematics and Statistics, Guangxi Normal University, 541004, Guilin, P.R. China
	\\Email address: jiabinyin@126.com

\end{document}